\RequirePackage{fix-cm}
\documentclass[smallextended]{svjour3}       
\smartqed  
\usepackage{amssymb,amsmath,multirow}
\usepackage{xcolor}
\usepackage{graphicx}
\usepackage{epstopdf}

\newtheorem{thm}{Theorem}
\newtheorem{rem}{Remark}
\newtheorem{exam}{Example}

\begin{document}

\title{An efficient global optimization algorithm for maximizing the sum of two generalized Rayleigh quotients
\thanks{This research was supported by National Natural Science Foundation of China under grants
11471325 and 11571029.} }


\titlerunning{Maximizing the sum of two Rayleigh quotients}        

\author{Xiaohui Wang  \and Longfei Wang\and Yong  Xia}


\institute{X. Wang \at  School of Astronautics,  Beihang University,  Beijing,
 100191, P. R. China \email{xhwang@buaa.edu.cn}
           \and
           L.F. Wang \and  Y. Xia \at \at
State Key Laboratory of Software Development Environment, LMIB of
the Ministry of Education, School of Mathematics and System
Sciences, Beihang University, Beijing 100191, P. R. China  \email{kingdflying999@126.com (L.F. Wang); dearyxia@gmail.com (Y. Xia)}
 }

\date{Received: date / Accepted: date}
\maketitle
\begin{abstract}
Maximizing the sum of two generalized Rayleigh quotients (SRQ) can be reformulated as a one-dimensional optimization problem, where the function value evaluations are reduced to solving semi-definite programming (SDP) subproblems. In this paper, we first use the optimal value of the dual SDP subproblem to construct a new saw-tooth-type overestimation. Then, we propose an efficient branch-and-bound algorithm to globally solve (SRQ),
which is shown to find an $\epsilon$-approximation optimal solution of (SRQ) in at most O$\left(\frac{1}{\epsilon}\right)$ iterations. Numerical results demonstrate that it is even more efficient than the recent SDP-based heuristic algorithm.

\keywords{: fractional programming, Rayleigh quotient,
semidefinite programming, branch and bound. }
 \subclass{ 90C32 \and  90C26 \and  90C22}
\end{abstract}

\section{Introduction}
The problem of maximizing the sum of two generalized Rayleigh quotients
\begin{align}\label{sum-of-two}
{\rm(SRQ)}~~ \max_{x\ne0}\dfrac{x^TBx}{x^TWx}+\dfrac{x^TDx}{x^TVx}
\end{align}
with positive definite matrices $W$ and $V$,
has recent applications in the multi-user MIMO
system \cite{PG} and the sparse Fisher discriminant analysis in
pattern recognition \cite{DFB,FM,WZW}. Without loss of generality, we can assume that $V$ is identity. Otherwise, we reformulate (\ref{sum-of-two}) as a problem in terms of $y$ by substituting $x=V^{-\frac{1}{2}}y$.
Moreover, since the objective function in (\ref{sum-of-two}) is homogeneous, (SRQ) can be further recast as the following sphere-constrained optimization problem, which is first  proposed by Zhang \cite{Hong,HZ}:
\begin{equation*}
{\rm(P)}~~ \begin{array}{lll}
&\max_{x\in \Bbb R^n}& f(x)
=\dfrac{x^TBx}{x^TWx}+x^TDx\\
&{\rm s.t.} & \|x\|= 1,
\end{array}
\end{equation*}
where $\|\cdot\|$ denotes the $\ell_2$-norm throughout this paper.

The single generalized Rayleigh quotient optimization problem (i.e., (SRQ) with $B=0$) is related to the classical eigenvalue problem and solved in polynomial time \cite{ZYL}.
However, to our best knowledge, whether the general (SRQ) (or (P)) can be efficiently solved in polynomial time remains open. Actually, as shown in [\cite{Hong}, Example 1.1], there could exist a few local  non-global maximizers of (P). Moreover, even finding the critical point of (P) is nontrivial, see  \cite{Hong,HZ}.

Recently,  (P) is reformulated as the problem of maximizing the following one-dimensional function \cite{NRX}:
\begin{equation}
({\rm P}_1)~~\max_{\mu\in [\underline{\mu},\bar{\mu}]}~q(\mu):=\mu+g(\mu), \label{oned}
\end{equation}
where $g(\mu)$ is related to a non-convex quadratic optimization:
\begin{equation}\label{P_mu}
\begin{array}{lll}
g(\mu)=&\max_{x\in \Bbb R^n}&x^TDx\\
&{\rm s.t.} & \|x\|= 1\\
&\ & x^T(B-\mu W)x\ge0
\end{array}
\end{equation}
and the lower and upper bounds
\begin{equation}\label{mu-bar}
\underline{\mu}=\min_{\|x\|=1} \dfrac{x^TBx}{x^TWx}, \ \
\bar{\mu}=\max_{\|x\|=1} \dfrac{x^TBx}{x^TWx}
\end{equation}
are the smallest and the largest generalized eigenvalues of the
matrix pencil $(B,W),$ respectively. In order to solve the one-dimensional problem (\ref{oned}),  a ``two-stage" heuristic algorithm is proposed in \cite{NRX} by first subdividing
$[\underline{\mu},\bar{\mu}]$ into coarse intervals such that each one contains a local maximizer of $q(\mu)$ and then applying the quadratic fit line search
\cite{An,Baz,Lu} in each interval.
For any given $\mu$, $g(\mu)$ (or $q(\mu)$) can be evaluated by solving an equivalent semi-definite programming (SDP) formulation, according to an extended version of S-Lemma in [\cite{Poly}, Proposition 4.1, see also [\cite{D}, Theorem 5.17]]. Finally,
for the returned optimal solution $\mu^*$, the optimal vector solution of (P) is recovered by a rank-one decomposition procedure [\cite{NRX}, Theorem 3]. Though this ``two-stage" algorithm could find the global solutions of the tested examples,  it is still a heuristic algorithm since the function $q(\mu)$ is not guaranteed to be quasi-concave. Besides, there is no meaningful stopping criterion for the ``two-stage" algorithm. That is, we cannot estimate the gap between the obtained solution and the global maximizer of (P$_1$).

In this paper, we propose an easy-to-evaluate  function for upper bounding $q(\mu)$. It provides saw-tooth-curve upper bounds of $q(\mu)$ over $[\underline{\mu},\bar{\mu}]$, which are used to establish an efficient branch-and-bound algorithm. We further show that the new algorithm
returns an $\epsilon$-approximation optimal solution of (P$_1$)
in at most $O\left(\frac{1}{\epsilon}\right)$ iterations. Numerical results show that the new algorithm is even much more efficient than the ``two-stage" heuristic algorithm \cite{NRX}.

The remainder of this paper is organized as follows. In Section 2, we give some preliminaries on the evaluation of $g(\mu)$. In Section 3, we propose an easy-to-compute upper bounding function, which provides saw-tooth-curve upper bounds of $g(\mu)$. In Section 4, we establish a new branch-and-bound algorithm and estimate the worst-case
computational complexity. In Section 5,
we do numerical comparison experiments, which demonstrate the efficiency of our new algorithm.
Conclusions are made in Section 6.

Throughout the paper, $v(\cdot)$ denotes the optimal objective value of the problem $(\cdot)$. We use $A\succeq (\preceq)0$ to stand for a positive (negative) semi-definite matrix $A$. The positive definite matrix $A$ is denoted by $A\succ 0$. Let $\lambda_{\max}(A)$ and $\lambda_{\min}(A)$ be the maximal and minimal eigenvalue of $A$, respectively. The inner product of two matrices $A$ and $B$ is denoted by $A\bullet B=$trace$(AB^T)$. For a real number $a$, $\lfloor a\rfloor$ returns the largest integer less than or equal to $a$.

\section{Preliminaries}

In the section, we first show how to evaluate $g(\mu)$. Then, we present the ``two-stage" algorithm \cite{NRX} to maximize $q(\mu)$ (\ref{oned}). Finally, we discuss how to get the optimal vector solution of (P) from the maximizer of $q(\mu)$.

Lifting $xx^T$ to $X\in \Bbb R^{n\times n}$ (since $x^TAx=A\bullet (xx^T)$) yields
the primal SDP relaxation of the optimization problem of evaluating $(g_{\mu})$ for any given $\mu$:
\begin{eqnarray*}
({\rm SDP}_{\mu})~~
&\max & D \bullet X\\
&{\rm s.t.} & I\bullet X= 1\\
& & (B-\mu W)\bullet X\ge0\\
& & X\succeq 0.
\end{eqnarray*}
The
conic dual problem of $({\rm SDP}_{\mu})$ is
\begin{eqnarray*}
({\rm SD}_{\mu})~~ &\min &\nu\\
&{\rm s.t.} &D -\nu I+\eta(B-\mu W)\preceq0\\
& & \eta\ge0,
\end{eqnarray*}
which coincides with the
Lagrangian dual problem of $g({\mu})$.

It is trivial to see that $({\rm SD}_{\mu})$ has an interior feasible solution, i.e., the Slater's condition holds. We can verify that,
for any $\mu$ satisfying
\begin{equation}
\mu<\bar{\mu},\label{ass}
\end{equation}
 the Slater's condition holds for $({\rm SDP}_{\mu})$, i.e., there is an $X\succ 0$ such that $I\bullet X= 1$ and
$(B-\mu W)\bullet X>0$. Therefore, under the assumption (\ref{ass}), strong duality holds for $({\rm SDP}_{\mu})$, that is, $v({\rm SDP}_{\mu})=v({\rm SD}_{\mu})$ and both optimal values are attained.

Under the assumption (\ref{ass}), by further
applying the extended version of S-Lemma in [\cite{Poly}, Proposition 4.1, see also [\cite{D}, Theorem 5.17]], we can show that the strong duality holds for the optimization problem of evaluating $g({\mu})$, i.e.,
$g({\mu})=v({\rm SD}_{\mu})$. For more details, we refer to \cite{NRX}.

Next, we present the ``two-stage" algorithm proposed in \cite{NRX} for solving (\ref{oned}). Firstly, it partitions
$[\underline{\mu},\bar{\mu}]$ into a rather coarse mesh and then collects all subintervals containing an interior local maximizer.
In the second stage, the quadratic fit method
\cite{Baz,An,Lu} is applied to find a corresponding local maximizer in each subinterval that has been collected in the first stage. Finally,  the optimal solution $\mu^*$ is selected from all these obtained local maximizers. In this paper, we will not present the detailed quadratic fit line search subroutine, which can be found in \cite{NRX}. One of the reason is that the algorithm in the first stage is already quite time-consuming.

{\bf The ``two-stage" scheme proposed in \cite{NRX}}

\begin{enumerate}
\item[$\empty$]
\begin{itemize}
\item[Step 1.] Given $\delta>0.$
Let $\mu_0=\underline{\mu}$ and $\mu_i=\underline{\mu}+(i-1)\delta$ for $i=1,2,\ldots,\lfloor\frac{\bar{\mu}-\underline{\mu}}
{\delta}\rfloor+1$. If $\frac{\bar{\mu}-\underline{\mu}}
{\delta}$ is not an integer, set $\mu_k=\bar{\mu}$ for
$k=
\lfloor\frac{\bar{\mu}-\underline{\mu}}
{\delta}\rfloor+2$.
\item[Step 2.]  For $i=1,2,\ldots,$
collect all the three-point pattern
$[\mu_{i-1},\mu_i,\mu_{i+1}]$ such that $\max\{q(\mu_{i-1}),q(\mu_{i+1})\}\le q(\mu_i)$.
\item[Step 3.] Call the quadratic fit line search subroutine (with a smaller tolerance than $\delta$) to
find a corresponding local maximizer in
each three-point pattern $[\mu_{i-1},\mu_i,\mu_{i+1}]$.
\item[Step 4.] Select the best maximizer $\mu^*$ among $\underline{\mu}$, $\bar{\mu}$, and all the local maximizers found in Step 3.
\end{itemize}
\end{enumerate}

Suppose (\ref{oned}) is solved, let $\mu^*$ be the returned maximizer.  If $\mu^*=\bar{\mu}$, the feasible region of (\ref{P_mu}) is reduced to
\[
\|x\|=1,~(B-\mu^*W)x=0,
\]
which contains only the unit eigenvector corresponding to the maximal eigenvalue. In this case, $g(\mu^*)$ is actually a maximum eigenvalue problem. On the other hand, suppose $\mu^*<\bar{\mu}$,
 the optimal vector solution of (P) is recovered from the equivalent $({\rm SDP}_{\mu^*})$ based on the rank one constraint, by using a rank-one procedure similar to
that in \cite{SZ,Y}, see details in \cite{NRX}.

 There is an alternative approach to recover the optimal solution of (P). Let $(\nu^*,\eta^*)$ be the optimal solution of the dual problem $({\rm SD}_{\mu^*})$. It is not difficult to verify that
\[
g(\mu^*)=\max_{\|x\|= 1}~ x^T(D-\eta^*(B-\mu^* W))x=\lambda_{\max}(D-\eta^*(B-\mu^* W)).
\]
Consequently, the optimal vector solution of (P) is the unit eigenvector corresponding to the maximum eigenvalue of $D-\eta^*(B-\mu^* W)$.

\section{Saw-tooth upper bounds}
In this section, we propose an easy-to-evaluate upper bounding function, which provides saw-tooth upper bounds for $q(\mu)$ over $[\underline{\mu},\bar{\mu}]$.

Let $\cup_{i=1}^k[\mu_i,\mu_{i+1}]$ be a partition of $[\underline{\mu},\bar{\mu}]$, where $\mu_1=\underline{\mu}$ and $\mu_{k+1}=\bar{\mu}$.

Consider the interval $[\mu_i,\mu_{i+1}]$ with $i\le k-1$ (so that
 $\mu_{i+1}<\bar{\mu}$).
Solve $({\rm SD}_{\mu})$ with $\mu=\mu_i,\mu_{i+1}$ and denote the optimal solutions by $(\nu_i,\eta_i)$ and $(\nu_{i+1},\eta_{i+1})$, respectively. Then, we have $\eta_{i}\ge 0$, $\eta_{i+1}\ge0$, and
\[
q(\mu_{i})=\mu_{i}+\nu_{i},~
q(\mu_{i+1})=\mu_{i+1}+\nu_{i+1}.
\]
For any $\mu\in[\mu_i,\mu_{i+1}]$, it follows from  the strong duality that
\begin{eqnarray}
q(\mu)&=&\mu+ \min_{\eta\ge0}\max_{\|x\|= 1} x^TDx+\eta(x^T(B-\mu W)x)\nonumber\\
&\le& \mu+ \max_{\|x\|= 1} x^TDx+\eta_i(x^T(B-\mu W)x)\nonumber\\
&=& \mu_i+ \max_{\|x\|= 1} \{x^TDx+\eta_i(x^T(B-\mu_i W)x)+\mu-\mu_i
+\eta_i(\mu_i-\mu)  x^TWx\}\nonumber\\
&\leq& q(\mu_i)+\mu-\mu_i
+\eta_i(\mu_i-\mu)  \max_{\|x\|= 1} x^TWx\nonumber\\
&\le& q(\mu_i)+\mu-\mu_i
+\eta_i(\mu_i-\mu)\min_{\|x\|= 1}x^TWx\nonumber\\
&=& q(\mu_i)+\mu-\mu_i
+\eta_i(\mu_i-\mu)\lambda_{\min}(W) \label{ub1}\\
&:=& q_1(\mu). \label{ub11}
\end{eqnarray}
Similarly, we have
\begin{equation}
q(\mu) \le q(\mu_{i+1})+\mu-\mu_{i+1}
+\eta_{i+1}(\mu_{i+1}-\mu)\lambda_{\max}(W)
:=q_2(\mu).
\label{ub2}
\end{equation}
Now, we obtain an upper bounding function of $q(\mu)$ over $[\mu_i,\mu_{i+1}]$:
\begin{equation}
\bar{q}(\mu)= \min\{q_1(\mu),q_2(\mu)\},
\label{ub}
\end{equation}
which is a concave function as $q_1(\mu)$ and $q_2(\mu)$ are both linear functions.
It provides the following upper bound of $q(\mu)$ over $[\mu_i,\mu_{i+1}]$:
\begin{equation}
U_i=\max_{\mu\in[\mu_i,\mu_{i+1}]} \bar{q}(\mu).
\label{ubb}
\end{equation}
Problem (\ref{ubb}) is a convex program. Moreover, it has a closed-form solution.
\begin{thm}\label{thm:b}
Under the assumption $\mu_{i+1}<\bar{\mu}$, an upper bound of $q(\mu)$ over $[\mu_i,\mu_{i+1}]$ is given by
\begin{eqnarray}
U_i=\left\{
\begin{array}{ll}
q(\mu_i),& {\rm if}~\eta_i\lambda_{\min}(W)\ge1       \label{bise1}\\
q(\mu_{i+1}),& {\rm if}~\eta_{i+1}\lambda_{\max}(W)\le1  \label{bise2}\\
q_1(\mu_0),& {\rm otherwise,}   \label{adapt}\\
\end{array}
\right.
\end{eqnarray}
where
\[
\mu_0=\frac{q(\mu_{i+1})-\mu_{i+1}
+\eta_{i+1}\mu_{i+1}\lambda_{\max}(W)-
q(\mu_i)+\mu_i
-\eta_i\mu_i\lambda_{\min}(W)}
{\eta_{i+1}\lambda_{\max}(W)-
\eta_i\lambda_{\min}(W)}.
\]
\end{thm}
\begin{proof}
The trivial proof is omitted as both $q_1(\mu)$ and $q_2(\mu)$ are linear functions and $\mu_0$ is the unique solution of the equation $q_1(\mu)=q_2(\mu)$.
\end{proof}

Finally, we also have a simple estimation of the upper bound $U_i$.
\begin{thm}\label{thm:b2}
For any $\mu\ge \mu_i$, we have
\begin{equation}
q(\mu) \le q(\mu_{i})+\mu-\mu_{i}.
\label{ub:es}
\end{equation}
\end{thm}
\begin{proof}
The inequality (\ref{ub:es}) follows from the definition $q_1(\mu)$ (\ref{ub11}) and the facts that $\eta_i\ge0$ and $\lambda_{\min}(W)>0$ (as $W\succ 0$).
\end{proof}
\begin{rem}
The estimation (\ref{ub:es}) is independent of $\mu_{i+1}$. Therefore, it can be satisfied for the extended case $\mu_{i+1}=\bar{\mu}$.
\end{rem}

\section{A saw-tooth branch-and-bound algorithm}

In this section, we first propose a branch-and-bound algorithm based on  the  new saw-tooth-curve upper bounds and then establish the worst-case computational complexity of the new algorithm.

Our algorithm works on a list
\begin{equation}
\underline{\mu}=\mu_1<\cdots<\mu_{k+1}=\bar{\mu}.
\label{list}
\end{equation}
The initial list is $\underline{\mu}=\mu_1<\mu_2=\bar{\mu}$.
In each iteration, we first
select the interval $[\mu_i,\mu_{i+1}]$ from the $\{\mu\}$-list that provides the maximal upper bound $U_i$ (\ref{ubb}). Then, we insert the mid-point $\frac{\mu_i+\mu_{i+1}}{2}$ into the $\{\mu\}$-list (\ref{list}) and increase  $k$ by one. The process is repeated until the stopping criterion is reached. The detailed algorithm is presented as follows.

{\bf The saw-tooth branch-and-bound algorithm}
\begin{enumerate}
\item[$\empty$]
\begin{itemize}
\item[Step 0.]
Given the approximation error $\epsilon>0$.
Compute $\underline{\mu}$, $\bar{\mu}$ (\ref{mu-bar}), $\lambda_{\min}(W)$ and $\lambda_{\max}(W)$. Initialize the iteration number $k=1$.

Let $\mu_1=\underline{\mu}$. Solve (${\rm SD}_{\mu_1}$) to obtain the optimal solution  $(\nu_1,\eta_1)$. Then, $q(\mu_{1})=\mu_1+\nu_1$ and let $LB=q(\mu_{1})$, $\mu^*=\mu_1$.

Let $\mu_2=\bar{\mu}-\epsilon$. If $\mu_2\le \underline{\mu}$, stop and return
$\mu^*$ as an approximate maximizer. Otherwise,
solve (${\rm SD}_{\mu_2}$) to obtain the optimal solution $(\nu_2,\eta_2)$. Then, $q(\mu_{2})=\mu_2+\nu_2$.
If $q(\mu_{2})>LB$,
update $LB=q(\mu_{2})$ and $\mu^*=\mu_2$. Set $k=2$  and $S=\emptyset$.

\item[Step 1.] Let $\tilde \mu= \frac{1}{2}(\mu_{1}+\mu_{2})$.
Solve (${\rm SD}_{\tilde\mu}$) and obtain the optimal solution  $(\tilde\nu,\tilde\eta)$. Then, $q(\tilde\mu)=\tilde\mu+\tilde\nu$.
If $q(\tilde \mu) > LB$, update $LB= q(\tilde \mu)$
 and $\mu^*=\tilde \mu$.
\item[Step 2.] According to Theorem \ref{thm:b},
compute the upper bounds:
\begin{eqnarray*}
UB_1= \max_{\mu\in [\mu_{1},\tilde \mu]} \bar{q}(\mu),~
UB_2= \max_{\mu\in [\tilde \mu,\mu_{2}]} \bar{q}(\mu).
\end{eqnarray*}
Update $S=S\cup\{(UB_1,\mu_{1},\tilde \mu)\}
\cup\{(UB_2,\tilde \mu,\mu_{2})\}$ and $k=k+1$.

\item[Step 3] Find $(UB^*,\mu_1,\mu_2)= \arg
\max\limits_{(t,*,*)\in S} t$.
   If $UB^*\leq LB+\epsilon$,  stop and return
$\mu^*$ as an approximate maximizer. Otherwise,  update
$S=S\setminus\{(UB^*,\mu_1,\mu_2)\}$ and go to Step 1.
\end{itemize}
\end{enumerate}

Theoretically, we can show that our new algorithm returns an
$\epsilon$-approximation optimal solution of (P$_1$) in at most
$O(\frac{1}{\epsilon})$ iterations. Here, we call $\mu^*$ an $\epsilon$-approximation
 optimal solution of (P$_1$)  if it is feasible and satisfies
\[
v({\rm P}_1)\ge q(\mu^*) \ge v({\rm P_1})-\epsilon.
\]
\begin{thm}\label{thm:comp}
The above algorithm terminates in at most $\left\lceil\frac{\bar{\mu}-\underline{\mu}}
{\epsilon}\right\rceil$ steps and returns
an $\epsilon$-approximation optimal solution of (P$_1$).
\end{thm}
\begin{proof}
If the algorithm terminates at Step 0, that is,
\[
\bar{\mu}-\epsilon \le \underline{\mu},
\]
then for any $\mu\in[\underline{\mu},\bar{\mu}]$, it follows from the inequality (\ref{ub:es}) in Theorem \ref{thm:b2} that
\[
q(\mu)\le  q(\underline{\mu})+\mu-\underline{\mu}
\le
q(\underline{\mu})+\bar{\mu}-\underline{\mu}
\le q(\underline{\mu})+\epsilon.
\]
Therefore, we have
\[
q(\mu^*)=q(\underline{\mu})\ge \max_{\mu\in[\underline{\mu},\bar{\mu}]}
q(\mu)-\epsilon
=v({\rm P}_1)-\epsilon.
\]
It follows that $\mu^*=\underline{\mu}$ is an $\epsilon$-approximation
 optimal solution of (P$_1$).

Now, we suppose that the algorithm does not terminate at Step 0.
Consider $\{(UB,\mu_1,\mu_{2})\}\in S$ in the $k$-th iteration of the algorithm.
If $UB<UB^*$, then the interval $[\mu_1,\mu_{2}]$ will be not selected to partition.
In the following, we assume $UB=UB^*$.
According to the inequality (\ref{ub:es}) in Theorem \ref{thm:b2},
for any $\mu\in[\mu_1,\mu_2]$, we have
\[
UB\le  q(\mu_1)+\mu_2-\mu_1.
\]
Since $UB=UB^*$ and $q(\mu_1)\le LB$, according to the stopping criterion, the algorithm terminates when
\[
\mu_2-\mu_1\le \epsilon.
\]
Therefore, there are at most
$\left\lceil\frac{\bar{\mu}-\underline\mu}{\epsilon}\right\rceil$ elements in $S$.
Since the number of elements of $S$ increases by one
in each iteration, the algorithm stops in $\left\lceil\frac{\bar{\mu}-\underline\mu}
{\epsilon}\right\rceil$ steps.

Let $\mu^*$ be the approximation solution returned by the algorithm. We have
\begin{equation}
UB^*\leq q(\mu^*)+\epsilon. \label{ULB}
\end{equation}
To show that $\mu^*$ is an $\epsilon$-approximation optimal solution of (P$_1$), it is sufficient to prove that
\begin{equation}
q(\mu^*)\ge v({\rm P}_1)-\epsilon.\label{appr}
\end{equation}
Let $\hat{\mu}=\bar{\mu}-\epsilon>\underline{\mu}$. According to the inequality (\ref{ub:es}) in Theorem \ref{thm:b2}, for any $\mu\in[\hat{\mu},\bar{\mu}]$, we obtain
\[
q(\mu)\le  q(\hat\mu)+\mu-\hat{\mu}\le  q(\hat\mu)+\bar\mu-\hat{\mu}=
q(\hat\mu)+\epsilon.
\]
Therefore, we have
\begin{eqnarray}
v({\rm P}_1)&\le& \max\{UB^*,\max\limits_{\mu\in[\hat{\mu},\bar{\mu}]}
q(\mu)\}\nonumber\\
&\le& \max\{UB^*, q(\hat\mu)+\epsilon \}\nonumber\\
&\le&q(\mu^*)+\epsilon.\label{eqq}
\end{eqnarray}
where the equality (\ref{eqq}) follows from (\ref{ULB}). Then, we obtain (\ref{appr}). The proof is complete.
\end{proof}

\section{Computational Experiments}\label{sec:3}

We test the new branch-and-bound algorithm  for solving (P$_1$) on the same numerical examples as in \cite{NRX}. The SDP subproblems $({\rm SD}_{\mu})$ are solved by
SDPT3 within CVX \cite{Boyd}. Since there is no unified stopping criterion in the
 ``two-stage'' heuristic algorithm \cite{NRX}, we just report the number of function evaluations (i.e., solving the SDP subproblems) in the first stage, with the setting $\delta=0.05$ used in \cite{NRX}. For our algorithm, we set $\epsilon=1e-5$.


The first example is taken from [\cite{Hong}, Example 3.2]. It has many local non-global maximizers.
\begin{exam}\label{exam3}
Let $B=\left(\begin{matrix}2.3969& 0.4651 &4.6392\\ 0.4651& 5.4401& 0.7838\\ 4.6392& 0.7838 &10.1741\end{matrix}\right),\\
W=
\left(\begin{matrix}0.8077& 0.8163& 1.0970\\ 0.8163 &4.1942& 0.8457\\ 1.0970& 0.8457& 1.8810\end{matrix}\right),
D=\left(\begin{matrix}3.9104& -0.9011& -2.0128\\ -0.9011& 0.9636& 0.6102\\-2.0128 & 0.6102& 1.0908\end{matrix}\right).$
\end{exam}
In this case, $[\underline{\mu},\bar{\mu}]=[0.9882,6.7322]$.  The ``two-stage'' algorithm \cite{NRX} gives an approximation solution $\mu^*=6.5952.$
The number of function evaluations in the first stage is $116$. Our algorithm returns an $\epsilon$-approximation optimal solution, $\mu^*=6.5952$, in $141$ iterations.

The second example in \cite{NRX} is  taking from [\cite{Hong}, Example 3.1], where
the optimal solution of $({\rm P}_1)$
is achieved at the right-hand side end-point $\bar{\mu}$.
\begin{exam}\label{exam2}
$B={\rm diag}(1, 9,2),W=D={\rm diag}(5,2,3).$
\end{exam}
In this case, $[\underline{\mu},\bar{\mu}]=[0.2,4.5]$. The number of function evaluations in the first stage of the ``two-stage'' algorithm \cite{NRX} is $87$. While our algorithm finds $\mu^*=4.5$ in $2$ iterations.

\begin{exam}[\cite{NRX}, Example 3]\label{exam4} Let
$$B=\left(\begin{matrix}1& 2 &3&1\\
2&5&4&-1\\3&4&0&1\\1&-1&1&6\end{matrix}\right), W= {\rm diag}(2
,1,5,10 ),
D=\left(\begin{matrix}5&-1&0&3\\-1&9&1&0\\0&1&-2&0\\3&0&0&8\end{matrix}\right).$$
\end{exam} In this case, $[\underline{\mu},\bar{\mu}]=[-0.8241,6.0647].$ The ``two-stage'' algorithm \cite{NRX} gives an approximation solution $\mu^*=5.8748.$
The number of function evaluations in the first stage is $139$. Our algorithm returns an $\epsilon$-approximation optimal solution, $\mu^*=5.8821$, in $35$ iterations.

\begin{exam}[\cite{NRX}, Example 4]\label{exam5} Let
 $n=10, B={\rm diag}(1, 2, 8, 7, 9, 3, 10, 2, -1, 6),\\
  W={\rm diag}( 9, 8, 7, 6, 5, 4, 3, 2, 1, 10), D={\rm diag}(5, 20, 3, 4, 8, -1, 0, 6, 32, 10).$
\end{exam}
The searching interval is $[\underline{\mu},\bar{\mu}]=[-1,3.3333].$
The optimal solution is the left-hand side end-point $-1$. The number of function evaluations in the first stage is $88$. Our algorithm returns an $\epsilon$-approximation optimal solution, $\mu^*=-1$, in $18$ iterations.

\begin{exam}[\cite{NRX}, Example 5]\label{exam1} Let
 $n=20,$\\
 $B={\rm diag}(1, 2, 20, 3, 50, 4, 6, 7, 8, 9, 100, 2, 3, 4, 5, 6, 7, 0, 10, 9);$\\
 $W={\rm diag}(100, 1, 2, 30, 5, 7, 9, 7, 8, 9, 1, 2, 30, 1, 50, 8, 1, 10, 10, 9);$\\
$D={\rm diag}(0, 1000, 20, 2, 5, 6, 7, 9, 50, 3, 4, 5, 100, 5, 2, 200, 4, 5, 9, 21).$
\end{exam}
The searching interval of this example is $[\underline{\mu},\bar{\mu}]=[0,100].$
The ``two-stage'' algorithm \cite{NRX} gives an approximation solution $\mu^*=2.0029.$
The number of function evaluations in the first stage is $2001$. Our algorithm returns an $\epsilon$-approximation optimal solution, $\mu^*=1.9999$, in $22$ iterations.

In addition to Examples 2-5 reported above, our algorithm highly outperforms the ``two-stage'' algorithm \cite{NRX}. For Example 1, our algorithm is also competitive. Notice that our algorithm is an exact algorithm and
the ``two-stage'' algorithm \cite{NRX} is heuristic.

Finally, we test more examples where the data are chosen randomly as follows.
Each component of the symmetric matrices $B$ and $D$ is uniformly distributed in $[-10,10]$. We generate $W,V=LL^T+\delta I$, where $L$ is a randomly generated lower bi-diagonal matrix with each nonzero element being uniformly distributed in $[-10,10]$ and $\delta>0$ is a constant number to guarantee the positive definiteness of $W$ and $V$.  For each dimension  varying from $30$ to $200$, we independently run the ``two-stage'' algorithm \cite{NRX} and our new algorithm ten times and report in Table \ref{tab} the average numerical results including the time in seconds and the number of iterations. It follows from the limited numerical results that our new global optimization algorithm highly outperforms the  ``two-stage'' heuristic algorithm.

\begin{table}[h]\centering
\caption{
The average of the numerical results for ten times solving
 (P) with different $n$.}
\label{tab}
\begin{tabular}{p{0.3cm}p{0.01cm}p{1.0cm}p{0.8cm}p{0.01cm}p{0.8cm}p{0.6cm}p{0.01cm}p{0.8cm}p{1.5cm}}
\hline
\multirow{2}{*}{n}&\multirow{2}{*}{}&\multicolumn{2}{c}{``two-stage'' algorithm  \cite{NRX}}&&
\multicolumn{2}{c}{Our new algorithm}\\
\cline{3-4}
\cline{6-7}
&&time(s)&iter.&&time(s)&iter.\\
\hline
30	&	& 58.84  &   233.6 &	 &	11.93	&	50.1  \\
50	&	& 98.19  &   320.8 &	 &	16.80	&	58.6  \\
80	&	& 192.09 &	 400.9 &	 &	31.59	&	68.7  \\
100	&	& 299.23 &	 459.3 &	 &	44.08	&	71.4  \\
120	&	& 493.83 &	 536.9 &	 &	62.52	&	71.3  \\
150	&	& 915.29 &	 609.4 &	 &	108.95	&	75.8  \\
180	&	& 1519.09&	 634.0 &	 &  186.84	&	81.2  \\
200	&	& 2118.18&	 672.2 &	 &  262.78	&	86.6  \\
\hline
\end{tabular}
\end{table}

\section{Conclusions}
The recent SDP-based heuristic algorithm for maximizing the sum of two generalized Rayleigh quotients (SRQ) is based on the one-dimensional parametric reformulation where each functional evaluation corresponds to solving a semi-definite programming (SDP) subproblem. In this paper,  we propose an efficient branch-and-bound algorithm to globally solve (SRQ) based on the new-developed saw-tooth overestimating approach. 
It is shown to find an $\epsilon$-approximation optimal solution of (SRQ) in at most O$\left(\frac{1}{\epsilon}\right)$ iterations. Numerical results demonstrate that it is much more efficient than the recent SDP-based heuristic algorithm.

\end{document}